\documentclass[10pt,reqno]{amsart}
\usepackage{amssymb}
\usepackage{amscd}
\usepackage{hyperref}

\newtheorem*{equivariantindextheorem}{Equivariant Index Theorem}
\newtheorem*{basicindextheorem}{Basic Index Theorem}
\theoremstyle{plain}
\newtheorem{theorem}{Theorem}

\newtheorem{lemma}[theorem]{Lemma}

\theoremstyle{definition}
\newtheorem{definition}[theorem]{Definition}

\newtheorem{remark}[theorem]{Remark}

\begin{document}
\title[The equivariant index]{The equivariant index theorem for
transversally elliptic operators and the basic index theorem for Riemannian
foliations}
\author[J.~Br\"{u}ning]{Jochen Br\"{u}ning}
\address{Institut f\"{u}r Mathematik \\
Humboldt Universit\"{a}t zu Berlin \\
Unter den Linden 6 \\
D-10099 Berlin, Germany}
\email[J.~Br\"{u}ning]{bruening@mathematik.hu-berlin.de}
\author[F. W.~Kamber]{Franz W.~Kamber}
\address{Department of Mathematics, University of Illinois \\
1409 W. Green Street \\
Urbana, IL 61801, USA}
\email[F. W.~Kamber]{kamber@math.uiuc.edu}
\author[K.~Richardson]{Ken Richardson}
\address{Department of Mathematics \\
Texas Christian University \\
Fort Worth, Texas 76129, USA}
\email[K.~Richardson]{k.richardson@tcu.edu}
\thanks{Work of the first author was partly supported by the grant SFB647 ``Space-Time-Matter''.}
\date{\today }

\begin{abstract}
In this expository paper, we explain a formula for the multiplicities of the
index of an equivariant transversally elliptic operator on a $G$-manifold.
The formula is a sum of integrals over blowups of the strata of the group
action and also involves eta invariants of associated elliptic operators.
Among the applications is an index formula for basic Dirac operators on
Riemannian foliations, a problem that was open for many years.
\end{abstract}

\maketitle

\subsection*{1. Introduction}


In this note we announce two new results in index theory, namely an
equivariant index theorem for transversally elliptic operators relative to a
compact Lie group action and the basic index theorem for transversal Dirac
operators in Riemannian foliations. The latter has been a well--known open
problem in foliation theory for more than twenty years. Complete proofs of
these results appear in \cite{BKR1} and \cite{BKR2}.

Suppose that a compact Lie group $G$ acts by isometries on a compact,
connected Riemannian manifold $M$, and let $E=E^{+}\oplus E^{-}$ be a
graded, $G$-equivariant Hermitian vector bundle over $M$. We consider a
first order $G$-equivariant differential operator $D=D^{+}:$ $\Gamma \left(
M,E^{+}\right) \rightarrow \Gamma \left( M,E^{-}\right) $ that is
transversally elliptic, and let $D^{-}$ be the formal adjoint of $D^{+}$.

The group $G$ acts on $\Gamma \left( M,E^{\pm }\right) $ by $\left(
gs\right) \left( x\right) =g\cdot s\left( g^{-1}x\right) $, and the
(possibly infinite-dimensional) subspaces $\ker \left( D^{+}\right) $ and $%
\ker \left( D^{-}\right) $ are $G$-invariant subspaces. Let $\rho
:G\rightarrow U\left( V_{\rho }\right) $ be an irreducible unitary
representation of $G$, and let $\chi _{\rho }=\mathrm{tr}\left( \rho \right) 
$ denote its character. Let $\Gamma \left( M,E^{\pm }\right) ^{\rho }$ be
the subspace of sections that is the direct sum of the irreducible $G$%
-representation subspaces of $\Gamma \left( M,E^{\pm }\right) $ that are
unitarily equivalent to the representation $\rho $. It can be shown that the
extended operators 
\begin{equation*}
\overline{D}_{\rho ,s}:H^{s}\left( \Gamma \left( M,E^{+}\right) ^{\rho
}\right) \rightarrow H^{s-1}\left( \Gamma \left( M,E^{-}\right) ^{\rho
}\right)
\end{equation*}%
are Fredholm and independent of $s$, so that each irreducible representation
of $G$ appears with finite multiplicity in $\ker D^{\pm }$. Let $a_{\rho
}^{\pm }\in \mathbb{Z}_{\geq 0}$ be the multiplicity of $\rho $ in $\ker
\left( D^{\pm }\right) $.

The study of index theory for such transversally elliptic operators was
initiated by M. Atiyah and I. Singer in the early 1970s (\cite{A}). The
virtual representation-valued index of $D$ is given by 
\begin{equation*}
\mathrm{ind}^{G}\left( D\right) :=\sum_{\rho }\left( a_{\rho }^{+}-a_{\rho
}^{-}\right) \left[ \rho \right] ,
\end{equation*}%
where $\left[ \rho \right] $ denotes the equivalence class of the
irreducible representation $\rho $. The index multiplicity is 
\begin{equation*}
\mathrm{ind}^{\rho }\left( D\right) :=a_{\rho }^{+}-a_{\rho }^{-}=\frac{1}{%
\dim V_{\rho }}\mathrm{ind}\left( \left. D\right\vert _{\Gamma \left(
M,E^{+}\right) ^{\rho }\rightarrow \Gamma \left( M,E^{-}\right) ^{\rho
}}\right) .
\end{equation*}%
In particular, if $\rho _{0}$ is the trivial representation of $G$, then 
\begin{equation*}
\mathrm{ind}^{\rho _{0}}\left( D\right) =\mathrm{ind}\left( \left.
D\right\vert _{\Gamma \left( M,E^{+}\right) ^{G}\rightarrow \Gamma \left(
M,E^{-}\right) ^{G}}\right) ,
\end{equation*}%
where the superscript $G$ implies restriction to $G$-invariant sections.

Atiyah's distributional index can be expanded in terms of the index
multiplicities, which play the role of generalized Fourier coefficients 
\begin{equation*}
\mathrm{ind}_{\ast }\left( D\right) \left( \phi \right) =\sum_{\rho }\mathrm{%
ind}^{\rho }\left( D\right) \int_{G}\phi \left( g\right) ~\overline{\chi
_{\rho }\left( g\right) }~dg\ .
\end{equation*}%
From this formula, we see that the multiplicities determine the
distributional index. Conversely, let $\alpha :G\rightarrow U\left(
V_{\alpha }\right) $ be an irreducible unitary representation. Then 
\begin{equation*}
\mathrm{ind}_{\ast }\left( D\right) \left( \chi _{\alpha }\right)
=\sum_{\rho }\mathrm{ind}^{\rho }\left( D\right) \int_{G}\chi _{\alpha
}\left( g\right) \overline{\chi _{\rho }\left( g\right) }\,dg=\mathrm{ind}%
^{\alpha }D,
\end{equation*}%
so that in principle complete knowledge of the distributional index is
equivalent to knowing all of the multiplicities $\mathrm{ind}^{\rho }\left(
D\right) $. Because the operator $\left. D\right\vert _{\Gamma \left(
M,E^{+}\right) ^{\rho }\rightarrow \Gamma \left( M,E^{-}\right) ^{\rho }}$
is Fredholm, all the indices $\mathrm{ind}^{\rho }\left( D\right) $ depend
only on the stable homotopy class of the principal transverse symbol of $D$.

Consider now the heat kernel expression for the index multiplicities. The
usual McKean-Singer argument shows that, in particular, for every $t>0$, the
index $\mathrm{ind}^{\rho }\left( D\right) $ may be expressed as the
following iterated integral: 
\begin{gather}
\mathrm{ind}^{\rho }\left( D\right) =\int_{x\in M}\int_{g\in G}\,\mathrm{str~%
}g\cdot K\left( t,g^{-1}x,x\right) ~\overline{\chi _{\rho }\left( g\right) }%
~dg~\left\vert dx\right\vert  \notag \\
=\int_{x\in M}\int_{g\in G}\left( \mathrm{tr~}g\cdot K^{+}\left(
t,g^{-1}x,x\right) -\mathrm{tr~}g\cdot K^{-}\left( t,g^{-1}x,x\right)
\right) ~\overline{\chi _{\rho }\left( g\right) }~dg~\left\vert dx\right\vert
\label{indexIntegralIntroduction}
\end{gather}%
where $K^{\pm }\left( t,\cdot ,\cdot \right) \in \Gamma \left( M\times
M,E^{\mp }\boxtimes \left( E^{\pm }\right) ^{\ast }\right) $ is the kernel
for $e^{-t\left( D^{\mp }D^{\pm }+C-\lambda _{\rho }\right) }$ on $\Gamma
\left( M,E^{\pm }\right) $, letting $\left\vert dx\right\vert $ denote the
Riemannian density over $M$.

A priori, the integral above is singular near sets of the form 
\begin{equation*}
\bigcup\limits_{G_{x}\in \left[ H\right] }x\times G_{x}\subset M\times G,
\end{equation*}%
where the isotropy subgroup $G_{x}$ is the subgroup of $G$ that fixes $x\in
M $, and $\left[ H\right] $ is a conjugacy class of isotropy subgroups.

Over the last twenty years numerous papers have appeared that express $%
\mathrm{ind}_{\ast }\left( D\right) $ and 
\begin{equation*}
\int_{M}\,\left( \mathrm{tr~}g\cdot K^{+}\left( t,g^{-1}x,x\right) -\mathrm{%
tr~}g\cdot K^{-}\left( t,g^{-1}x,x\right) \right) ~\left\vert dx\right\vert
\end{equation*}%
in terms of topological and geometric quantities, as in the
Atiyah-Segal-Singer index theorem for elliptic operators \cite{ASe} or the
Berline-Vergne Theorem for transversally elliptic operators \cite{Be-V1},%
\cite{Be-V2}. However, until now there has been very little known about the
problem of expressing $\mathrm{ind}^{\rho }\left( D\right) $ in terms of
topological or geometric quantities which are determined at the different
strata of the $G$-manifold $M$. The special case when all of the isotropy
groups are the same dimension was solved by M. Atiyah in \cite{A}, and this
result was utilized by T. Kawasaki to prove the Orbifold Index Theorem (see 
\cite{Kawas2}). Our analysis is new in that the integral over the group in (%
\ref{indexIntegralIntroduction}) is performed first, before integration over
the manifold, and thus the invariants in our index theorem are very
different from those seen in other equivariant index formulas.

Our main theorem (Theorem \ref{MainTheorem}) expresses $\mathrm{ind}^{\rho
}\left( D\right) $ as a sum of integrals over the different strata of the
action of $G$ on $M$, and it involves the eta invariant of associated
equivariant elliptic operators on spheres normal to the strata. The result
is the following.

\begin{equivariantindextheorem}
The equivariant index $\mathrm{ind}^{\rho }\left( D\right)$ is given by the
formula 
\begin{eqnarray*}
\mathrm{ind}^{\rho }\left( D\right) &=&\int_{G\diagdown \widetilde{M_{0}}%
}A_{0}^{\rho }\left( x\right) ~\widetilde{\left\vert dx\right\vert }%
~+\sum_{j=1}^{r}\beta \left( \Sigma _{\alpha _{j}}\right) ~, \\
\beta \left( \Sigma _{\alpha _{j}}\right) &=&\frac{1}{2\dim V_{\rho }}%
\sum_{b\in B}\frac{1}{n_{b}\mathrm{rank~}W^{b}}\left( -\eta \left(
D_{j}^{S+,b}\right) +h\left( D_{j}^{S+,b}\right) \right) \int_{G\diagdown 
\widetilde{\Sigma _{\alpha _{j}}}}A_{j,b}^{\rho }\left( x\right) ~\widetilde{%
\left\vert dx\right\vert }~.
\end{eqnarray*}
\end{equivariantindextheorem}

The notation will be explained later; e.g. the integrands $A_{0}^{\rho
}\left( x\right) $ and $A_{j,b}^{\rho }\left( x\right) $ are the familar
Atiyah-Singer integrands corresponding to local heat kernel supertraces of
induced elliptic operators over closed manifolds. Even in the case when the
operator $D$ is elliptic, this result was not known previously. Further, the
formula above gives a method for computing eta invariants of Dirac-type
operators on quotients of spheres by compact group actions; these have been
computed previously only in some special cases. We emphasize that every part
of the formula is explicitly computable from local information provided by
the operator and manifold. Even the eta invariant of the operator $%
D_{j}^{S+,b}$ on a sphere is calculated directly from the principal symbol
of the operator $D$ at one point of a singular stratum. Examples show that
all of the terms in the formula above are nontrivial.

The de Rham operator provides an important example illustrating the
computability of the formula, yielding a new theorem expressing the
equivariant Euler characteristic in terms of ordinary Euler characteristics
of the strata of the group action (Theorem \ref{EulerCharacteristicTheorem}).

One of the primary motivations for obtaining an explicit formula for $%
\mathrm{ind}^{\rho }\left( D\right) $ was to use it to produce a basic index
theorem for Riemannian foliations, thereby solving a problem that has been
open since the 1980s. In fact the basic index theorem is a consequence of
the invariant index theorem corresponding to the trivial representation $%
\rho _{0}$. This theorem is stated below, with more details in Section 6. 

\begin{basicindextheorem}
The basic index is given by the formula 
\begin{eqnarray*}
\mathrm{ind}_{b}\left( D_{b}^{E}\right) &=&\int_{\widetilde{M_{0}}\diagup 
\overline{\mathcal{F}}}A_{0,b}\left( x\right) ~\widetilde{\left\vert
dx\right\vert }+\sum_{j=1}^{r}\beta \left( M_{j}\right) ~ \\
\beta \left( M_{j}\right) &=&\frac{1}{2}\sum_{\tau }\frac{1}{n_{\tau }%
\mathrm{rank~}W^{\tau }}\left( -\eta \left( D_{j}^{S+,\tau }\right) +h\left(
D_{j}^{S+,\tau }\right) \right) \int_{\widetilde{M_{j}}\diagup \overline{%
\mathcal{F}}}A_{j,b}^{\tau }\left( x\right) ~\widetilde{\left\vert
dx\right\vert },
\end{eqnarray*}%
where the sum is over all components of singular strata and over all
canonical isotropy bundles $W^{\tau }$, only a finite number of which yield
nonzero terms $A_{j,b}^{\tau }$.
\end{basicindextheorem}

Several techniques in this paper are new and have not been previously
explored. First, the fact that $\mathrm{ind}^{\rho }\left( D\right) $ is
invariant under $G$-equivariant homotopies is used in a very specific way,
and we keep track of the effects of these homotopies so that the formula for
the index reflects data coming from the original operator and manifold. In
Section 4 
we describe a process of blowing up, cutting, and reassembling the $G$%
-manifold into what is called the desingularization, which also involves
modifying the operator and vector bundles near the singular strata as well.
The result is a $G$-manifold that has less intricate structure and for which
the heat kernels are easier to evaluate. The key idea is to relate the local
asymptotics of the equivariant heat kernel of the original manifold to the
desingularized manifold; at this stage the eta invariant appears through a
direct calculation on the normal bundle to the singular stratum. More
precisely, we compute the local contribution of the supertrace of a general
constant coefficient equivariant heat operator in the neighborhood of a
singular point of an orthogonal group action on a sphere. It is here that
the equivariant index is related to a boundary value problem, which explains
the presence of eta invariants in the main theorem.

Another new idea in this paper is the decomposition of equivariant vector
bundles over $G$-manifolds with one orbit type. A crucial step in the proof
required the construction of a subbundle of an equivariant bundle over a $G$%
-invariant part of a stratum that is the minimal $G$-bundle decomposition
that consists of direct sums of isotypical components of the bundle. We call
this decomposition the \emph{fine decomposition} and define it in Section 2. 
A more detailed account of this method will appear in \cite{KRi}.

The relevant properties of the supertrace of the equivariant heat kernel are
discussed in Section 3. 
We apply the heat kernel analysis, representation theory, and fine
decomposition to produce a heat kernel splitting formula. This process leads
to a reduction formula for the equivariant heat supertrace, from which the
Equivariant Index Theorem \ref{MainTheorem} follows. Examples show that all
the terms in the index formula are in general nontrivial.

We note that a recent paper
of Gorokhovsky and Lott addresses this transverse index question on Riemannian foliations.
Using a different technique, they are able to prove a formula for the basic index of a basic
Dirac operator that is distinct from our formula, in the case where all the infinitesimal
holonomy groups of the foliation are connected tori and if Molino's commuting sheaf is
abelian and has trivial holonomy (see \cite{GLott}).

We thank James Glazebrook, Efton Park and Igor Prokhorenkov for helpful
discussions. The authors would like to thank variously the Mathematisches
Forschungsinstitut Oberwolfach, the Erwin Schr\"{o}dinger International
Institute for Mathematical Physics (ESI), Vienna, the Department for
Mathematical Sciences (IMF) at Aarhus University, the Centre de Recerca Matem%
\`{a}tica (CRM), Barcelona, and the Department of Mathematics at TCU for
hospitality and support during the preparation of this work.

\subsection*{2. The refined isotypical decomposition}


Let $X$ be a smooth Riemannian manifold on which a compact Lie group $G$
acts by isometries with single orbit type $\left[ H\right] $ (see Section
4). Let $X^{H}$ be the fixed point set of $H$, and for $\alpha \in \pi
_{0}\left( X^{H}\right) $, let $X_{\alpha }^{H}$ denote the corresponding
connected component of $X^{H}$. Let $N=N\left( H\right) $ be the normalizer.

\begin{definition}
\label{componentRelGDefn}We denote $X_{\alpha }=GX_{\alpha }^{H}$, and $%
X_{\alpha }$ is called a \textbf{component of} $X$ \textbf{relative to} $G$.
\end{definition}

\begin{remark}
The space $X_{\alpha }$ is not necessarily connected, but it is the inverse
image of a connected component of $G\diagdown X=N\diagdown X^{H}$ under the
projection $X\rightarrow G\diagdown X$. Also, note that $X_{\alpha
}=X_{\beta }$ if there exists $n\in N$ such that $nX_{\alpha }^{H}=X_{\beta
}^{H}$. If $X$ is a closed manifold, then there are a finite number of
components of $X$ relative to $G$.
\end{remark}

We now introduce a decomposition of a $G$-bundle $E\rightarrow X$ over a $G$%
-space with single orbit type $\left[ H\right] $. Let $E_{\alpha }$ be the
restriction $\left. E\right\vert _{X_{\alpha }^{H}}$. For $\sigma
:H\rightarrow U\left( W_{\sigma }\right) $ an irreducible unitary
representation, let $\sigma ^{n}:H\rightarrow U\left( W_{\sigma }\right) $
be the irreducible representation defined by%
\begin{equation*}
\sigma ^{n}\left( h\right) =\sigma \left( n^{-1}hn\right) .
\end{equation*}%
We let $\left. E_{\alpha }^{\left[ \sigma \right] }\right\vert _{X_{\alpha
}^{H}}$ denote the $\left[ \sigma \right] $-isotypical component of $E$ over 
$X_{\alpha }^{H}$, meaning that each $E_{\alpha ,x}^{\left[ \sigma \right] }$
is the subspace of $E_{x}$ that is a direct sum of irreducible $H$%
-representation subspaces of type $\left[ \sigma \right] $. We define 
\begin{equation*}
E_{\alpha ,\left[ \sigma \right] ,x}^{N}=\mathrm{span}\left\{ E_{\alpha ,x}^{%
\left[ \sigma ^{n}\right] }:n\in N\text{ and }nX_{\alpha }^{H}=X_{\alpha
}^{H}\right\} .
\end{equation*}
The $N$-orbits yield an $N$-bundle $E_{\alpha ,\left[ \sigma \right] }^{N}$
over $NX_{\alpha }^{H}\subseteq X^{H}$, and a similar bundle may be formed
over each distinct $NX_{\beta }^{H}$, with $\beta \in \pi _{0}\left(
X^{H}\right) $. Further, observe that since each bundle $E_{\alpha ,\left[
\sigma \right] }^{N}$ is an $N$-bundle over $NX_{\alpha }^{H}$, it defines a
unique $G$ bundle $E_{\alpha ,\left[ \sigma \right] }^{G}$.

\begin{definition}
\label{fineComponentDefinition}The $G$-bundle $E_{\alpha ,\left[ \sigma %
\right] }^{G}$ over the submanifold $X_{\alpha }$ is called a \textbf{fine
component} or the \textbf{fine component of }$E\rightarrow X$ \textbf{%
associated to }$\left( \alpha ,\left[ \sigma \right] \right) $.
\end{definition}

If $G\diagdown X$ is not connected, one must construct the fine components
separately over each $X_{\alpha }$. If $E$ has finite rank, then $E$ may be
decomposed as a direct sum of distinct fine components over each $X_{\alpha
} $. In any case, $E_{\alpha ,\left[ \sigma \right] }^{N}$ is a finite
direct sum of isotypical components over each $X_{\alpha }^{H}$.

\begin{definition}
\label{FineDecompositionDefinition}The direct sum decomposition of $\left.
E\right\vert _{X_{\alpha }}$ into subbundles $E^{b}$ that are fine
components $E_{\alpha ,\left[ \sigma \right] }^{G}$ for some $\left[ \sigma %
\right] $, written 
\begin{equation*}
\left. E\right\vert _{X_{\alpha }}=\bigoplus\limits_{b}E^{b}~,
\end{equation*}%
is called the \textbf{refined isotypical decomposition} (or \textbf{fine
decomposition}) of $\left. E\right\vert _{X_{\alpha }}$.
\end{definition}

In the case where $G\diagdown X$ is connected, the group $\pi _{0}\left(
N\diagup H\right) $ acts transitively on the connected components $\pi
_{0}\left( X^{H}\right) $, and thus $X_{\alpha }=X$. We comment that if $%
\left[ \sigma ,W_{\sigma }\right] $ is an irreducible $H$-representation
present in $E_{x}$ with $x\in X_{\alpha }^{H}$, then $E_{x}^{\left[ \sigma %
\right] }$ is a subspace of a distinct $E_{x}^{b}$ for some $b$. The
subspace $E_{x}^{b}$ also contains $E_{x}^{\left[ \sigma ^{n}\right] }$ for
every $n$ such that $nX_{\alpha }^{H}=X_{\alpha }^{H}$~.

\begin{remark}
\label{constantMultiplicityRemark}Observe that by construction, for $x\in
X_{\alpha }^{H}$ the multiplicity and dimension of each $\left[ \sigma %
\right] $ present in a specific $E_{x}^{b}$ is independent of $\left[ \sigma %
\right] $. Thus, $E_{x}^{\left[ \sigma ^{n}\right] }$ and $E_{x}^{\left[
\sigma \right] }$ have the same multiplicity and dimension if $nX_{\alpha
}^{H}=X_{\alpha }^{H}$~.
\end{remark}

\begin{remark}
The advantage of this decomposition over the isotypical decomposition is
that each $E^{b}$ is a $G$-bundle defined over all of $X_{\alpha }$, and the
isotypical decomposition may only be defined over $X_{\alpha }^{H}$.
\end{remark}

\begin{definition}
\label{adaptedDefn}Now, let $E$ be a $G$-equivariant vector bundle over $X$,
and let $E^{b}~$be a fine component as in Definition \ref%
{fineComponentDefinition} corresponding to a specific component $X_{\alpha
}=GX_{\alpha }^{H}$ of $X$ relative to $G$. Suppose that another $G$-bundle $%
W$ over $X_{\alpha }$ has finite rank and has the property that the
equivalence classes of $G_{y}$-representations present in $E_{y}^{b},y\in
X_{\alpha }$ exactly coincide with the equivalence classes of $G_{y}$%
-representations present in $W_{y}$, and that $W$ has a single component in
the fine decomposition. Then we say that $W$ is \textbf{adapted} to $E^{b}$.
\end{definition}

\begin{lemma}
\label{AdaptedToAnyBundleLemma}In the definition above, if another $G$%
-bundle $W$ over $X_{\alpha }$ has finite rank and has the property that the
equivalence classes of $G_{y}$-representations present in $E_{y}^{b},y\in
X_{\alpha }$ exactly coincide with the equivalence classes of $G_{y}$%
-representations present in $W_{y}$, then it follows that $W$ has a single
component in the fine decomposition and hence is adapted to $E^{b}$. Thus,
the last phrase in the corresponding sentence in the above definition is
superfluous.
\end{lemma}

\begin{proof}
Suppose that we choose an equivalence class $\left[ \sigma \right] $ of $H$%
-representations present in $W_{x}$, $x\in X_{\alpha }^{H}$. Let $\left[
\sigma ^{\prime }\right] $ be any other equivalence class; then, by
hypothesis, there exists $n\in N$ such that $nX_{\alpha }^{H}=X_{\alpha
}^{H} $ and $\left[ \sigma ^{\prime }\right] =\left[ \sigma ^{n}\right] $.
Then, observe that $nW_{x}^{\left[ \sigma \right] }=W_{nx}^{\left[ \sigma
^{n}\right] }=W_{x}^{\left[ \sigma ^{n}\right] }$, with the last equality
coming from the rigidity of irreducible $H$-representations. Thus, $W$ is
contained in a single fine component, and so it must have a single component
in the fine decomposition.
\end{proof}

In what follows, we show that there are naturally defined finite-dimensional
vector bundles that are adapted to any fine components. 
We enumerate the irreducible representations $\left\{ \left[ \rho
_{j},V_{\rho _{j}}\right] \right\} _{j=1,2,...}$ of $G$. Let $\left[ \sigma
,W_{\sigma }\right] $ be any irreducible $H$-representation. Let $G\times
_{H}W_{\sigma } $ be the corresponding homogeneous vector bundle over the
homogeneous space $G\diagup H$. Then the $L^{2}$-sections of this vector
bundle decompose into irreducible $G$-representations. In particular, let $%
\left[ \rho _{j_{0}},V_{\rho _{j_{0}}}\right] $ be the equivalence class of
irreducible representations that is present in $L^{2}\left( G\diagup
H,G\times _{H}W_{\sigma }\right) $ and that has the lowest index $j_{0}$.
Then Frobenius reciprocity implies%
\begin{equation*}
0\neq \mathrm{Hom}_{G}\left( V_{\rho _{j_{0}}},L^{2}\left( G\diagup
H,G\times _{H}W_{\sigma }\right) \right) \cong \mathrm{Hom}_{H}\left( V_{%
\mathrm{\mathrm{Res}}\left( \rho _{j_{0}}\right) },W_{\sigma }\right) ,
\end{equation*}%
so that the restriction of $\rho _{j_{0}}$ to $H$ contains the $H$%
-representation $\left[ \sigma \right] $. Now, for a component $X_{\alpha
}^{H}$ of $X^{H}$, with $X_{\alpha }=GX_{\alpha }^{H}$ its component in $X$
relative to $G$, the trivial bundle%
\begin{equation*}
X_{\alpha }\times V_{\rho _{j_{0}}}
\end{equation*}%
is a $G$-bundle (with diagonal action) that contains a nontrivial fine
component $W_{\alpha ,\left[ \sigma \right] }$ containing $X_{\alpha
}^{H}\times \left( V_{\rho _{j_{0}}}\right) ^{\left[ \sigma \right] }$.

\begin{definition}
\label{canonicalIsotropyBundleDefinition}We call $W_{\alpha ,\left[ \sigma %
\right] }\rightarrow X_{\alpha }$ the \textbf{canonical isotropy }$G$\textbf{%
-bundle associated to }$\left( \alpha ,\left[ \sigma \right] \right) \in \pi
_{0}\left( X^{H}\right) \times \widehat{H}$. Observe that $W_{\alpha ,\left[
\sigma \right] }$ depends only on the enumeration of irreducible
representations of $G$, the irreducible $H$-representation $\left[ \sigma %
\right] $ and the component $X_{\alpha }^{H}$. 
\end{definition}

\begin{lemma}
\label{canIsotropyGbundleAdaptedExists}Given any $G$-bundle $E\rightarrow X$
and any fine component $E^{b}$ of $E$ over some $X_{\alpha }=GX_{\alpha
}^{H} $, there exists a canonical isotropy $G$-bundle $W_{\alpha ,\left[
\sigma \right] }$ adapted to $E^{b}\rightarrow X_{\alpha }$.
\end{lemma}

\subsection*{3. Equivariant heat kernel and equivariant index}


\medskip We now review some properties of the equivariant index and
equivariant heat kernel that are known to experts in the field (see \cite{A}%
, \cite{BrH1}, \cite{BrH2}, \cite{Be-G-V}). With notation as in the
introduction, let $E=E^{+}\oplus E^{-}$ be a graded, $G$-equivariant vector
bundle over $M$. We consider a first order $G$-equivariant differential
operator $D^{+}:$ $\Gamma \left( M,E^{+}\right) \rightarrow \Gamma \left(
M,E^{-}\right) $ which is transversally elliptic, and let $D^{-}$ be the
formal adjoint of $D^{+}$. The restriction $D^{\pm ,\rho }=\left. D^{\pm
}\right\vert _{\Gamma \left( M,E\right) ^{\rho }}$ behaves in a similar way
to an elliptic operator. Let $\left\{ X_{1},...,X_{r}\right\} $ be an
orthonormal basis of the Lie algebra of $G$. Let $\mathcal{L}_{X_{j}}$
denote the induced Lie derivative with respect to $X_{j}$ on sections of $E$%
, and let $C=\sum_{j}\mathcal{L}_{X_{j}}^{\ast }\mathcal{L}_{X_{j}}$ be the
Casimir operator on sections of $E$. Let $\lambda _{\rho }\geq 0$ be the
eigenvalue of $C$ associated to the representation type $\left[ \rho \right] 
$. The following argument can be seen in some form in \cite{A}. Given a
section $\alpha \in \Gamma \left( M,E^{+}\right) ^{\rho }$, we have 
\begin{equation*}
D^{-}D^{+}\alpha =\left( D^{-}D^{+}+C-\lambda _{\rho }\right) \alpha .
\end{equation*}%
Then $D^{-}D^{+}+C-\lambda _{\rho }$ is self-adjoint and elliptic and has
finite dimensional eigenspaces consisting of smooth sections. The $\left[
\rho \right] $-part $K^{\left[ \rho \right] }$ of the heat kernel of $%
e^{-tD^{-}D^{+}}$ is the same as the $\left[ \rho \right] $-part of the heat
kernel $K\left( t,\cdot ,\cdot \right) $ of $e^{-t\left(
D^{-}D^{+}+C-\lambda _{\rho }\right) }$. One can show that 
\begin{equation}
\mathrm{ind}^{\rho }\left( D\right) =\int_{x\in M}\int_{g\in G}\,\mathrm{str~%
}g\cdot K\left( t,g^{-1}x,x\right) ~\overline{\chi _{\rho }\left( g\right) }%
~dg~\left\vert dx\right\vert ,  \label{heatKernelExpressionForIndex}
\end{equation}%
where $\chi _{\rho }\left( g\right) $ Since the heat kernel $K$ changes
smoothly with respect to $G$-equivariant deformations of the metric and of
the operator $D$ and the right hand side is an integer, we see that $\mathrm{%
ind}^{\rho }\left( D\right) $ is stable under such homotopies of the
operator $D^{+}$ through $G$-equivariant transversally elliptic operators.
This implies that the indices $\mathrm{ind}^{G}\left( D\right) $ and $%
\mathrm{ind}_{g}\left( D\right) $ mentioned in the introduction depend only
on the $G$-equivariant homotopy class of the principal transverse symbol of $%
D^{+}$. Since the integral above is independent of $t$, we may compute its
asymptotics as $t\rightarrow 0^{+}$ to determine $\mathrm{ind}^{\rho }\left(
D\right) $.

One important idea is that the asymptotics of $K\left( t,g^{-1}x,x\right) $
as $t\rightarrow 0$ are completely determined by the operator's local
expression along the minimal geodesic connecting $g^{-1}x$ and $x$, if $%
g^{-1}x$ and $x$ are sufficiently close together. If the distance between $%
g^{-1}x$ and $x$ is bounded away from zero, there is a constant $c>0$ such
that $K\left( t,g^{-1}x,x\right) =\mathcal{O}\left( e^{-c/t}\right) $ as $%
t\rightarrow 0$. For these reasons, it is clear that the asymptotics of the
supertrace $\mathrm{ind}^{\rho }\left( D\right) $ are locally determined
over spaces of orbits in $M$, meaning that the contribution to the index is
a sum of $t^{0}$-asymptotics of integrals of $\int_{U}\int_{G}\mathrm{%
\mathrm{str}}\left( g\cdot K\left( t,g^{-1}w,w\right) \right) \,\overline{%
\chi _{\rho }\left( g\right) }~dg$ $\left\vert dw\right\vert $ over a finite
collection of saturated sets $U\subset M$ that are unions of orbits that
intersect a neighborhood of a point of $M$. However, the integral of the $%
t^{0}$ asymptotic coefficient of $\int_{G}\mathrm{\mathrm{str}}\left( g\cdot
K\left( t,g^{-1}w,w\right) \right) \overline{\chi _{\rho }\left( g\right) }%
~dg$ over all of $M$ is not the index, and the integral of the $t^{0}$
asymptotic coefficient of $\int_{M}\mathrm{\mathrm{str}\,}\left( g\cdot
K\left( t,g^{-1}w,w\right) \right) ~\mathrm{dvol}_{M}$ over all of $G$ is
not the index. Thus, the integrals over $M$ and $G$ may not be separated
when computing the local index contributions. In particular, the singular
strata of the group action may not be ignored.

\subsection*{4. Desingularizing along a singular stratum}


In the first part of this section, we will describe some standard results
from the theory of Lie group actions (see \cite{Bre}, \cite{Kaw}). Such $G$%
-manifolds are stratified spaces, and the stratification can be described
explicitly. As above, $G$ is a compact Lie group acting on a smooth,
connected, closed manifold $M$. We assume that the action is effective,
meaning that no $g\in G$ fixes all of $M$. (Otherwise, replace $G$ with $%
G\diagup \left\{ g\in G:gx=x\text{ for all }x\in M\right\} $.) Choose a
Riemannian metric for which $G$ acts by isometries; average the pullbacks of
any fixed Riemannian metric over the group of diffeomorphisms to obtain such
a metric.

For $x\in M$, the isotropy or stabilizer subgroup $G_{x}<G$ is defined to be 
$\left\{ g\in G:gx=x\right\} $. The orbit $\mathcal{O}_{x}$ of a point $x$
is defined to be $\left\{ gx:g\in G\right\} $. Since $G_{xg}=gG_{x}g^{-1}$,
the conjugacy class of the isotropy subgroup of a point is fixed along an
orbit.

On any such $G$-manifold, the conjugacy class of the isotropy subgroups
along an orbit is called the \textbf{orbit type}. On any such $G$-manifold,
there are a finite number of orbit types, and there is a partial order on
the set of orbit types. Given subgroups $H$ and $K$ of $G$, we say that $%
\left[ H\right] \leq $ $\left[ K\right] $ if $H$ is conjugate to a subgroup
of $K$, and we say $\left[ H\right] <$ $\left[ K\right] $ if $\left[ H\right]
\leq $ $\left[ K\right] $ and $\left[ H\right] \neq $ $\left[ K\right] $. We
may enumerate the conjugacy classes of isotropy subgroups as $\left[ G_{0}%
\right] ,...,\left[ G_{r}\right] $ such that $\left[ G_{i}\right] \leq \left[
G_{j}\right] $ implies that $i\leq j$. It is well-known that the union of
the principal orbits (those with type $\left[ G_{0}\right] $) form an open
dense subset $M_{0}$ of the manifold $M$, and the other orbits are called 
\textbf{singular}. As a consequence, every isotropy subgroup $H$ satisfies $%
\left[ G_{0}\right] \leq \left[ H\right] $. Let $M_{j}$ denote the set of
points of $M$ of orbit type $\left[ G_{j}\right] $ for each $j$; the set $%
M_{j}$ is called the \textbf{stratum} corresponding to $\left[ G_{j}\right] $%
. If $\left[ G_{j}\right] \leq \left[ G_{k}\right] $, it follows that the
closure of $M_{j}$ contains the closure of $M_{k}$. A stratum $M_{j}$ is
called a \textbf{minimal stratum} if there does not exist a stratum $M_{k}$
such that $\left[ G_{j}\right] <\left[ G_{k}\right] $ (equivalently, such
that $\overline{M_{k}}\subsetneq \overline{M_{j}}$). It is known that each
stratum is a $G$-invariant submanifold of $M$, and in fact a minimal stratum
is a closed (but not necessarily connected) submanifold. Also, for each $j$,
the submanifold $M_{\geq j}:=\bigcup\limits_{\left[ G_{k}\right] \geq \left[
G_{j}\right] }M_{k}$ is a closed, $G$-invariant submanifold. To prove the
main theorem of this paper, we decompose the $G$-manifold using tubular
neighborhoods of the minimal strata. \label{BlowUpSection}

We will now construct a new $G$-manifold $N$ that has a single stratum (of
type $\left[ G_{0}\right] $) and that is a branched cover of $M$, branched
over the singular strata. A distinguished fundamental domain of $M_{0}$ in $%
N $ is called the \textbf{desingularization of }$M$ and is denoted $%
\widetilde{M}$. We also refer to \cite{AlMel} for their recent related
explanation of this desingularization (which they call \emph{resolution}).
To simplify this discussion, we will assume that the codimension of each
singular stratum is at least two.

A sequence of modifications is used to construct $N$ and $\widetilde{M}%
\subset N$. Let $M_{j}$ be a minimal stratum. Let $T_{\varepsilon }\left(
M_{j}\right) $ denote a tubular neighborhood of radius $\varepsilon $ around 
$M_{j}$, with $\varepsilon $ chosen sufficiently small so that all orbits in 
$T_{\varepsilon }\left( M_{j}\right) \setminus M_{j}$ are of type $\left[
G_{k}\right] $, where $\left[ G_{k}\right] <\left[ G_{j}\right] $. Let 
\begin{equation*}
N^{1}=\left( M\setminus T_{\varepsilon }\left( M_{j}\right) \right) \cup
_{\partial T_{\varepsilon }\left( M_{j}\right) }\left( M\setminus
T_{\varepsilon }\left( M_{j}\right) \right)
\end{equation*}%
be the manifold constructed by gluing two copies of $\left( M\setminus
T_{\varepsilon }\left( M_{j}\right) \right) $ smoothly along the boundary.
Since the $T_{\varepsilon }\left( M_{j}\right) $ is saturated (a union of $G$%
-orbits), the $G$-action lifts to $N^{1}$. Note that the strata of the $G$%
-action on $N^{1}$ correspond to strata in $M\setminus T_{\varepsilon
}\left( M_{j}\right) $. If $M_{k}\cap \left( M\setminus T_{\varepsilon
}\left( M_{j}\right) \right) $ is nontrivial, then the stratum corresponding
to isotropy type $\left[ G_{k}\right] $ on $N^{1}$ is 
\begin{equation*}
N_{k}^{1}=\left( M_{k}\cap \left( M\setminus T_{\varepsilon }\left(
M_{j}\right) \right) \right) \cup _{\left( M_{k}\cap \partial T_{\varepsilon
}\left( M_{j}\right) \right) }\left( M_{k}\cap \left( M\setminus
T_{\varepsilon }\left( M_{j}\right) \right) \right) .
\end{equation*}%
Thus, $N^{1}$ is a $G$-manifold with one fewer stratum than $M$, and $%
M\setminus M_{j}$ is diffeomorphic to one copy of $\left( M\setminus
T_{\varepsilon }\left( M_{j}\right) \right) $, denoted $\widetilde{M}^{1}$
in $N^{1}$. In fact, $N^{1}$ is a branched double cover of $M$, branched
over $M_{j}$. If $N^{1}$ has one orbit type, then we set $N=N^{1}$ and $%
\widetilde{M}=\widetilde{M}^{1}$. If $N^{1}$ has more than one orbit type,
we repeat the process with the $G$-manifold $N^{1}$ to produce a new $G$%
-manifold $N^{2}$ with two fewer orbit types than $M$ and that is a $4$-fold
branched cover of $M$. Again, $\widetilde{M}^{2}$ is a fundamental domain of 
$\widetilde{M}^{1}\setminus \left\{ \text{a minimal stratum}\right\} $,
which is a fundamental domain of $M$ with two strata removed. We continue
until $N=N^{r}$ is a $G$-manifold with all orbits of type $\left[ G_{0}%
\right] $ and is a $2^{r}$-fold branched cover of $M$, branched over $%
M\setminus M_{0}$. We set $\widetilde{M}=\widetilde{M}^{r}$, which is a
fundamental domain of $M_{0}$ in $N$.

Further, one may independently desingularize $M_{\geq j}$, since this
submanifold is itself a closed $G$-manifold. If $M_{\geq j}$ has more than
one connected component, we may desingularize all components simultaneously.
The isotropy type of all points of $\widetilde{M_{\geq j}}$ is $\left[ G_{j}%
\right] $, and $\widetilde{M_{\geq j}}\diagup G$ is a smooth (open) manifold.

We now more precisely describe the desingularization. If $M$ is equipped
with a $G$-equivariant, transversally elliptic differential operator on
sections of an equivariant vector bundle over $M$, then this data may be
pulled back to the desingularization $\widetilde{M}$. Given the bundle and
operator over $N^{j}$, simply form the invertible double of the operator on $%
N^{j+1}$, which is the double of the manifold with boundary $N^{j}\setminus
T_{\varepsilon }\left( \Sigma \right) $, where $\Sigma $ is a minimal
stratum on $N^{j}$.

Specifically, we modify the metric equivariantly so that there exists $%
\varepsilon >0$ such that the tubular neighborhood $B_{2\varepsilon }\Sigma $
of $\Sigma $ in $N^{j}$ is isometric to a ball of radius $2\varepsilon $ in
the normal bundle $N\Sigma $. In polar coordinates, this metric is $%
ds^{2}=dr^{2}+d\sigma ^{2}+r^{2}d\theta _{\sigma }^{2}$, with $r\in \left(
0,2\varepsilon \right) $, $d\sigma ^{2}$ is the metric on $\Sigma $, and $%
d\theta _{\sigma }^{2}$ is the metric on $S\left( N_{\sigma }\Sigma \right) $%
, the unit sphere in $N_{\sigma }\Sigma $; note that $d\theta _{\sigma }^{2}$
is isometric to the Euclidean metric on the unit sphere. We simply choose
the horizontal metric on $B_{2\varepsilon }\Sigma $ to be the pullback of
the metric on the base $\Sigma $, the fiber metric to be Euclidean, and we
require that horizontal and vertical vectors be orthogonal. We do not assume
that the horizontal distribution is integrable.

Next, we replace $r^{2}$ with $f\left( r\right) =\left[ \tilde{g}\left(
r\right) \right] ^{2}$ in the expression for the metric, where $\tilde{g}$
is defined so that the metric is cylindrical for small $r$.

In our description of the modification of the differential operator, we will
need the notation for the (external) product of differential operators.
Suppose that $F\hookrightarrow X\overset{\pi }{\rightarrow }B$ is a fiber
bundle that is locally a metric product. Given an operator $A_{1,x}:\Gamma
\left( \pi ^{-1}\left( x\right) ,E_{1}\right) \rightarrow \Gamma \left( \pi
^{-1}\left( x\right) ,F_{1}\right) $ that is locally given as a differential
operator $A_{1}:\Gamma \left( F,E_{1}\right) \rightarrow \Gamma \left(
F,F_{1}\right) $ and $A_{2}:\Gamma \left( B,E_{2}\right) \rightarrow \Gamma
\left( B,F_{2}\right) $ on Hermitian bundles, we have the product 
\begin{equation*}
A_{1,x}\ast A_{2}:\Gamma \left( X,\left( E_{1}\boxtimes E_{2}\right) \oplus
\left( F_{1}\boxtimes F_{2}\right) \right) \rightarrow \Gamma \left(
X,\left( F_{1}\boxtimes E_{2}\right) \oplus \left( E_{1}\boxtimes
F_{2}\right) \right)
\end{equation*}%
as in K-theory (see, for example, \cite{A}, \cite[pp. 384ff]{LM}), which is
used to define the Thom Isomorphism in vector bundles.

Let $D=D^{+}:\Gamma \left( N^{j},E^{+}\right) \rightarrow \Gamma \left(
N^{j},E^{-}\right) $ be the given first order, transversally elliptic, $G$%
-equivariant differential operator. Let $\Sigma \ $be a minimal stratum of $%
N^{j}$. Here we assume that $\Sigma $ has codimension at least two. We
modify the metrics and bundles equivariantly so that there exists $%
\varepsilon >0$ such that the tubular neighborhood $B_{\varepsilon }\left(
\Sigma \right) $ of $\Sigma $ in $M$ is isometric to a ball of radius $%
\varepsilon $ in the normal bundle $N\Sigma $, and so that the $G$%
-equivariant bundle $E$ over $B_{\varepsilon }\left( \Sigma \right) $ is a
pullback of the bundle $\left. E\right\vert _{\Sigma }\rightarrow \Sigma $.
We assume that near $\Sigma $, after a $G$-equivariant homotopy $D^{+}$ can
be written on $B_{\varepsilon }\left( \Sigma \right) $ locally as the
product 
\begin{equation*}
D^{+}=\left( D_{N}\ast D_{\Sigma }\right) ^{+},
\end{equation*}%
where $D_{\Sigma }\ $is a transversally elliptic, $G$-equivariant, first
order operator on the stratum $\Sigma $, and $D_{N}$ is a $G$-equivariant,
first order operator on $B_{\varepsilon }\left( \Sigma \right) $ that is
elliptic on the fibers. If $r$ is the distance from $\Sigma $, we write $%
D_{N}$ in polar coordinates as%
\begin{equation*}
D_{N}=Z\left( \nabla _{\partial _{r}}^{E}+\frac{1}{r}D^{S}\right)
\end{equation*}%
where $Z=-i\sigma \left( D_{N}\right) \left( \partial _{r}\right) $ is a
local bundle isomorphism and the map $D^{S}$ is a purely first order
operator that differentiates in the unit normal bundle directions tangent to 
$S_{x}\Sigma $.

On first glance, the product form above appears to be restrictive, but as we
show in \cite{BKR1}, the product assumption is satisfied by most operators
under fairly weak topological conditions on the stratum.

We modify the operator $D_{N}$ on each Euclidean fiber of $N\Sigma \overset{%
\pi }{\rightarrow }\Sigma $ by converting the conical metric to a
cylindrical metric via a radial blow-up; the result is a $G$-manifold $%
\widetilde{M}^{j}$ with boundary $\partial \widetilde{M}^{j}$, a $G$-vector
bundle $\widetilde{E}^{j}$, and the induced operator $\widetilde{D}^{j}$,
all of which locally agree with the original counterparts outside $%
B_{\varepsilon }\left( \Sigma \right) $. We may double $\widetilde{M}^{j}$
along the boundary $\partial \widetilde{M}^{j}$ and reverse the chirality of 
$\widetilde{E}^{j}$ as described in \cite[Ch. 9]{Bo-Wo}. Doubling produces a
closed $G$-manifold $N^{j}$, a $G$-vector bundle $E^{j}$, and a first-order
transversally elliptic differential operator $D^{j}$. This process may be
iterated until all orbits of the resulting $G$-manifold are principal.

In the technical proof of the main theorem in \cite{BKR1}, we carefully
track the changes to the heat kernel integral (\ref%
{heatKernelExpressionForIndex}) throughout the desingularization process.
When the radial blowup occurs, the manifold is replaced with a manifold with
boundary and nonlocal boundary conditions; this is the reason that eta
invariants appear in the formula for the equivariant index multiplicities.
In spite of this, every part of the formula is explicitly computable from
the principal transverse symbol of the operator restricted to small
saturated neighborhoods.

The crucial formula is as follows. In calculating the small $t$ asymptotics
of 
\begin{equation*}
\int \mathrm{str}K\left( t,z_{p},z_{p}\right) ^{\rho }=\int \mathrm{\mathrm{%
str}}\left( E_{t}\left( D\right) ^{\rho }\right) \left( z_{p},z_{p}\right)\ ,
\end{equation*}
with $E_{t}\left( D\right) ^{\rho }=\exp \left( -tD^{\ast }D\right) ^{\rho }$%
, it suffices to calculate the right hand side of the formula above over a
small tubular neighborhood $B_{\varepsilon }\left( U\right) \subset M$ of a
saturated open set $U\subset \Sigma _{\alpha }\subseteq \Sigma $ in a most
singular stratum. We then sum over fine components $b\in B$ (see Definition %
\ref{fineComponentDefinition}), using the heat kernel coming from the blown
up manifold $\widetilde{B_{\varepsilon }\left( U\right) }$. As $t\rightarrow
0$, 
\begin{gather}
\int_{B_{\varepsilon }\left( U\right) }\mathrm{str}K\left(
t,z_{p},z_{p}\right) ^{\rho }\sim \int_{\widetilde{B_{\varepsilon }\left(
U\right) }}\mathrm{str}K\left( t,z_{p},z_{p}\right) ^{\rho }  \notag \\
+\sum_{b}\frac{1}{2n_{b}\mathrm{rank}\left( W^{b}\right) }\left( -\eta
\left( D^{S+,b}\right) +h\left( D^{S+,b}\right) \right) \int_{p\in U}\mathrm{%
str}K_{\Sigma }^{b}\left( t,p,p\right) ^{\rho },  \label{heatSplitting}
\end{gather}%
with $\mathrm{str}K_{\Sigma }^{b}\left( t,p,p\right) ^{\rho }=\mathrm{str}%
\left( E_{t}\left( \mathbf{1}^{b}\otimes D_{\Sigma }\right) ^{\rho }\right)
\left( p,p\right) $ is the local heat supertrace corresponding to the
operator $\mathbf{1}^{b}\otimes D_{\Sigma }$ on $\Gamma \left(
U,W^{b}\otimes E_{\Sigma }\right) ^{\rho }$. The eta invariant $\eta \left(
D^{S+,b}\right) $ is the equivariant eta invariant of $D^{S+}$ restricted to
isotropy representation types present in $W^{b}$; since the eigenvalues of $%
D^{S+}$ are integers, this quantity is constant over components of the
stratum relative to $G$. Similarly, the dimension $h\left( D^{S+,b}\right) $
of the kernel of $D^{S+}$ restricted to those sections is locally constant.

\subsection*{5. The Equivariant Index Theorem}


To evaluate $\mathrm{ind}^{\rho }\left( D\right) $ as in Equation (\ref%
{heatKernelExpressionForIndex}), we apply formula (\ref{heatSplitting})
repeatedly, starting with a minimal stratum and then applying to each double
of the equivariant desingularization. After all the strata are blown up and
doubled, all of the resulting manifolds have a single stratum, and the $G$%
-manifold is a fiber bundle with homogeneous fibers. We obtain the following
result. In what follows, if $U$ denotes an open subset of a stratum of the
action of $G$ on $M$, $U^{\prime }$ denotes the equivariant
desingularization of $U$, and $\widetilde{U}$ denotes the fundamental domain
of $U$ inside $U^{\prime }$, as in Section 4. 
We also refer the reader to Definitions \ref{componentRelGDefn} and \ref%
{canonicalIsotropyBundleDefinition}. For the sake of simplicity of
exposition, we assume that the codimension of each stratum is at least two.

\begin{theorem}
(Equivariant Index Theorem) \label{MainTheorem}Let $M_{0}$ be the principal
stratum of the action of a compact Lie group $G$ on the closed Riemannian $M$%
, and let $\Sigma _{\alpha _{1}}$,...,$\Sigma _{\alpha _{r}}$ denote all the
components of all singular strata relative to $G$. Let $E\rightarrow M$ be a
Hermitian vector bundle on which $G$ acts by isometries. Let $D:\Gamma
\left( M,E^{+}\right) \rightarrow \Gamma \left( M,E^{-}\right) $ be a first
order, transversally elliptic, $G$-equivariant differential operator. We
assume that near each $\Sigma _{\alpha _{j}}$, $D$ is $G$-homotopic to the
product $D_{N}\ast D^{\alpha _{j}}$, where $D_{N}$ is a $G$-equivariant,
first order differential operator on $B_{\varepsilon }\Sigma $ that is
elliptic and has constant coefficients on the fibers and $D^{\alpha _{j}}\ $%
is a global transversally elliptic, $G$-equivariant, first order operator on
the $\Sigma _{\alpha _{j}}$. In polar coordinates 
\begin{equation*}
D_{N}=Z_{j}\left( \nabla _{\partial _{r}}^{E}+\frac{1}{r}D_{j}^{S}\right) ~,
\end{equation*}%
where $r$ is the distance from $\Sigma _{\alpha _{j}}$, where $Z_{j}$ is a
local bundle isometry (dependent on the spherical parameter), the map $%
D_{j}^{S}$ is a family of purely first order operators that differentiates
in directions tangent to the unit normal bundle of $\Sigma _{j}$. Then the
equivariant index $\mathrm{ind}^{\rho }\left( D\right) $ is given by the
formula 
\begin{eqnarray*}
\mathrm{ind}^{\rho }\left( D\right)  &=&\int_{G\diagdown \widetilde{M_{0}}%
}A_{0}^{\rho }\left( x\right) ~\widetilde{\left\vert dx\right\vert }%
~+\sum_{j=1}^{r}\beta \left( \Sigma _{\alpha _{j}}\right) ~, \\
\beta \left( \Sigma _{\alpha _{j}}\right)  &=&\frac{1}{2\dim V_{\rho }}%
\sum_{b}\frac{1}{n_{b}\mathrm{rank~}W^{b}}\left( -\eta \left(
D_{j}^{S+,b}\right) +h\left( D_{j}^{S+,b}\right) \right) \int_{G\diagdown 
\widetilde{\Sigma _{\alpha _{j}}}}A_{j,b}^{\rho }\left( x\right) ~\widetilde{%
\left\vert dx\right\vert }~,
\end{eqnarray*}%
where the sum is over all canonical isotropy bundles $W^{b}$, a finite
number of which yield nonzero $A_{j,b}^{\rho }$, and where

\begin{enumerate}
\item $A_{0}^{\rho }\left( x\right) $ is the Atiyah-Singer integrand, the
local supertrace of the ordinary heat kernel associated to the elliptic
operator induced from $D^{\prime }$ (blown-up and doubled from $D$) on the
quotient $M_{0}^{\prime }\diagup G$, where the bundle $E$ is replaced by the
finite-dimensional space of sections of type $\rho $ over an orbit.

\item Similarly, $A_{i,b}^{\rho }$ is the local supertrace of the ordinary
heat kernel associated to the elliptic operator induced from $\left( \mathbf{%
1}\otimes D^{\alpha _{j}}\right) ^{\prime }$ (blown-up and doubled from $%
\mathbf{1}\otimes D^{\alpha _{j}}$, the twist of $D^{\alpha _{j}}$ by the
canonical isotropy bundle $W^{b}\rightarrow \Sigma _{\alpha _{j}}$ ) on the
quotient $\Sigma _{\alpha _{j}}^{\prime }\diagup G$, where the bundle is
replaced by the space of sections of type $\rho $ over each orbit.

\item $\eta \left( D_{j}^{S+,b}\right) $ is the eta invariant of the
operator $D_{j}^{S+}$ induced on any unit normal sphere $S_{x}\Sigma
_{\alpha _{j}}$, restricted to sections of isotropy representation types in $%
W_{x}^{b}$, which is constant on $\Sigma _{\alpha _{j}}$.

\item $h\left( D_{j}^{S+,b}\right) $ is the dimension of the kernel of $%
D_{j}^{S+,b}$, restricted to sections of isotropy representation types in $%
W_{x}^{b}$, again constant on on $\Sigma _{\alpha _{j}}$.

\item $n_{b}$ is the number of different inequivalent $G_{x}$-representation
types present in each $W_{x}^{b}$, $x\in \Sigma _{\alpha _{j}}$.
\end{enumerate}
\end{theorem}

We now give a simple application of our result. It is well known that if $M$
is a Riemannian manifold and $f:M\rightarrow M$ is an isometry that is
homotopic to the identity, then the Euler characteristic of $M$ is the sum
of the Euler characteristics of the fixed point sets of $f$. We generalize
this result as follows. We consider the de Rham operator 
\begin{equation*}
d+d^{\ast }:\Omega ^{\mathrm{even}}\left( M\right) \rightarrow \Omega ^{%
\mathrm{odd}}\left( M\right)
\end{equation*}%
on a $G$-manifold, and the invariant index of this operator is the
equivariant Euler characteristic $\chi ^{G}\left( M\right) $, the Euler
characteristic of the elliptic complex consisting of invariant forms. If $G$
is connected and the Euler characteristic is expressed in terms of its $\rho 
$-components, only the invariant part $\chi ^{G}\left( M\right) =\chi ^{\rho
_{0}}\left( M\right) $ appears. This is a consequence of the homotopy
invariance of de Rham cohomology. Thus $\chi ^{G}\left( M\right) =\chi
\left( M\right) $ for connected Lie groups $G$. In general the Euler
characteristic is a sum of components%
\begin{equation*}
\chi \left( M\right) =\sum_{\left[ \rho \right] }\chi ^{\rho }\left(
M\right) ,
\end{equation*}%
where $\chi ^{\rho }\left( M\right) $ is the alternating sum of the
dimensions of the $\left[ \rho \right] $-parts of the cohomology groups (or
spaces of harmonic forms). Since the connected component $G_{0}$ of the
identity in $G$ acts trivially on the harmonic forms, the only nontrivial
components $\chi ^{\rho }\left( M\right) $ correspond to representations
induced from unitary representations of the finite group $G\diagup G_{0}$.

Using Theorem \ref{MainTheorem} and the formula $\mathrm{ind}^{\rho}\left(
d+d^{\ast }\right) = \frac{1}{\dim V_{\rho }}\chi ^{\rho }\left(M\right) $,
we obtain the following result.

\begin{theorem}
\label{EulerCharacteristicTheorem}Let $M$ be a compact $G$-manifold, with $G$
a compact Lie group and principal isotropy subgroup $H_{\mathrm{pr}}$. Let $%
M_{0}$ denote the principal stratum, and let $\Sigma _{\alpha _{1}}$,...,$%
\Sigma _{\alpha _{r}}$ denote all the components of all singular strata
relative to $G$. Then 
\begin{eqnarray*}
\chi ^{\rho }\left( M\right) &=&\chi ^{\rho }\left( G\diagup H_{\mathrm{pr}%
}\right) \chi \left( G\diagdown M,G\diagdown \text{singular strata}\right) \\
&&+\sum_{j}\chi ^{\rho }\left( G\diagup G_{j}\text{~},\mathcal{L}%
_{N_{j}}\right) \chi \left( G\diagdown \overline{\Sigma _{\alpha _{j}}}%
,G\diagdown \text{lower strata}\right) ,
\end{eqnarray*}%
where $\mathcal{L}_{N_{j}}$ is the orientation line bundle of normal bundle
of the stratum component $\Sigma _{\alpha _{j}}$.
\end{theorem}

In the formula above, the $\chi \left( X,Y\right) $ refers to the relative
Euler characteristic (see \cite{BKR1} for details and examples).

\subsection*{6. The Basic Index Theorem for Riemannian foliations}


\medskip The content of this section is discussed and proved in detail in 
\cite{BKR2}. Let $M$ be an $n$-dimensional, closed, connected manifold, and
let $\mathcal{F}$ be a codimension $q$ foliation on $M$. Let $Q$ denote the
quotient bundle $TM\diagup T\mathcal{F}$ over $M$. Such a foliation is
called a \emph{Riemannian foliation} if it is endowed with a metric on $Q$
(called the transverse metric) that is \emph{holonomy-invariant}; that is,
the Lie derivative of that transverse metric with respect to every leafwise
tangent vector is zero. The metric on $Q$ can always be extended to a
Riemannian metric on $M$; the extended metric restricted to the normal
bundle $N\mathcal{F}=\left( T\mathcal{F}\right) ^{\bot }$ agrees with the
transverse metric via the isomorphism $Q\cong N\mathcal{F}$. We refer the
reader to \cite{Mo}, \cite{Re}, and \cite{T} for introductions to the
geometric and analytic properties of Riemannian foliations.

Let $\widehat{M}$ be the transverse orthonormal frame bundle of $(M,\mathcal{%
F})$, and let $p$ be the natural projection $p:\widehat{M}\rightarrow M$.
The Bott connection is a natural connection on $Q$ that induces a connection
on $\widehat{M}$ (see \cite[pp. 80ff]{Mo} ). The manifold $\widehat{M}$ is a
principal $O(q)$-bundle over $M$. Given $\hat{x}\in \widehat{M}$, let $\hat{x%
}g$ denote the well-defined right action of $g\in G=O(q)$ applied to $\hat{x}
$. Associated to $\mathcal{F}$ is the lifted foliation $\widehat{\mathcal{F}}
$ on $\widehat{M}$; the distribution $T\widehat{\mathcal{F}}$ is the
horizontal lift of $T\mathcal{F}$. By the results of Molino (see \cite[%
pp.~105-108, p.~147ff]{Mo} ), the lifted foliation is transversally
parallelizable (meaning that there exists a global basis of the normal
bundle consisting of vector fields whose flows preserve $\widehat{\mathcal{F}%
}$), and the closures of the leaves are fibers of a fiber bundle $\widehat{%
\pi }:\widehat{M}\rightarrow \widehat{W}$. The manifold $\widehat{W}$ is
smooth and is called the basic manifold. Let $\overline{\widehat{\mathcal{F}}%
}$ denote the foliation of $\widehat{M}$ by leaf closures of $\widehat{%
\mathcal{\ F}}$, which is shown by Molino to be a fiber bundle. The leaf
closure space of $\left( M,\mathcal{F}\right) $ is denoted $W=M\diagup 
\overline{\mathcal{F}}=\widehat{W}\diagup G$. 
\begin{equation*}
\begin{array}{ccccccc}
p^{\ast }E &  &  &  & \mathcal{E} &  &  \\ 
& \searrow  &  &  & \downarrow  &  &  \\ 
O\left( q\right)  & \hookrightarrow  & \left( \widehat{M},\widehat{\mathcal{F%
}}\right)  & \overset{\widehat{\pi }}{\longrightarrow } & \widehat{W} &  & 
\\ 
&  & \downarrow ^{p} & \circlearrowleft  & \downarrow \,\, &  &  \\ 
E & \rightarrow  & \left( M,\mathcal{F}\right)  & \overset{\pi }{%
\longrightarrow } & W &  & 
\end{array}%
\end{equation*}

Endow $(\widehat{M},\widehat{\mathcal{F}})$ with the transverse metric $%
g^{Q}\oplus g^{O(q)}$, where $g^{Q}$ is the pullback of metric on $Q$, and $%
g^{O(q)}$ is the standard, normalized, biinvariant metric on the fibers. We
require that vertical vectors are orthogonal to horizontal vectors. This
transverse metric gives each of $(\widehat{M},\widehat{\mathcal{F}})$ and $(%
\widehat{M},\overline{\widehat{\mathcal{F}}})$ the structure of a Riemannian
foliation. The transverse metric on $(\widehat{M},\overline{\widehat{%
\mathcal{F}}})$ induces a well--defined Riemannian metric on $\widehat{W}$.
The action of $G=O(q)$ on $\widehat{M}$ induces an isometric action on $%
\widehat{W}$.

For each leaf closure $\overline{\widehat{L}}\in \overline{\widehat{\mathcal{%
F}}}$ and $\widehat{x}\in \overline{\widehat{L}}$, the restricted map $p:%
\overline{\widehat{L}}\rightarrow \overline{L}$ is a principal bundle with
fiber isomorphic to a subgroup $H_{\widehat{x}}\subset O(q)$, which is the
isotropy subgroup at the point $\widehat{\pi }(\hat{x})\in \widehat{W}$. The
conjugacy class of this group is an invariant of the leaf closure $\overline{%
L}$, and the strata of the group action on $\widehat{W}$ correspond to the
strata of the leaf closures of $(M,\mathcal{F})$.

A \emph{basic form} over a foliation is a global differential form that is
locally the pullback of a form on the leaf space; more precisely, $\alpha
\in \Omega ^{\ast }\left( M\right) $ is basic if for any vector tangent to
the foliation, the interior product with both $\alpha $ and $d\alpha $ is
zero. A\ \emph{basic vector field} is a vector field $V$ whose flow
preserves the foliation. In a Riemannian foliation, near any point it is
possible to choose a local orthonormal frame of $Q$ represented by basic
vector fields.

A vector bundle $E\rightarrow \left( M,\mathcal{F}\right) $ that is \emph{%
foliated} may be endowed with a basic connection $\nabla ^{E}$ (one for
which the associated curvature forms are basic -- see \cite{KT2}). An
example of such a bundle is the normal bundle $Q$. Given such a foliated
bundle, a section $s\in \Gamma \left( E\right) $ is called a \emph{basic
section} if for every $X\in T\mathcal{F}$, $\nabla _{X}^{E}s=0$. Let $\Gamma
_{b}\left( E\right) $ denote the space of basic sections of $E$. Note that
the basic sections of $Q$ correspond to basic normal vector fields.

An example of another foliated bundle over a component of a stratum $M_{j}$
is the bundle defined as follows. Let $E\rightarrow M$ be any foliated
vector bundle. Let $\Sigma _{\alpha _{j}}=\widehat{\pi }\left( p^{-1}\left(
M_{j}\right) \right) $ be the corresponding stratum on the basic manifold $%
\widehat{W}$, and let $W^{\tau }\rightarrow \Sigma _{\alpha _{j}}$ be a
canonical isotropy bundle (Definition \ref{canonicalIsotropyBundleDefinition}%
). Consider the bundle $\widehat{\pi }^{\ast }W^{\tau }\otimes p^{\ast
}E\rightarrow p^{-1}\left( M_{j}\right) $, which is foliated and basic for
the lifted foliation restricted to $p^{-1}\left( M_{j}\right) $. This
defines a new foliated bundle $E^{\tau }\rightarrow M_{j}$ by letting $%
E_{x}^{\tau }$ be the space of $O\left( q\right) $-invariant sections of $%
\widehat{\pi }^{\ast }W^{\tau }\otimes p^{\ast }E$ restricted to $%
p^{-1}\left( x\right) $. We call this bundle \textbf{the} $W^{\tau }$\textbf{%
-twist of} $E\rightarrow M_{j}$.

Suppose that $E$ is a foliated $\mathbb{C}\mathrm{l}\left( Q\right) $ module
with basic $\mathbb{C}\mathrm{l}\left( Q\right) $ connection $\nabla ^{E}$
over a Riemannian foliation $\left( M,\mathcal{F}\right) $. Then it can be
shown that Clifford multiplication by basic vector fields preserves $\Gamma
_{b}\left( E\right) $, and we have the operator%
\begin{equation*}
D_{b}^{E}:\Gamma _{b}\left( E^{+}\right) \rightarrow \Gamma _{b}\left(
E^{-}\right)
\end{equation*}%
defined for any local orthonormal frame $\left\{ e_{1},...,e_{q}\right\} $
for $Q$ by%
\begin{equation*}
D_{b}^{E}=\left. \sum_{j=1}^{q}c\left( e_{j}\right) \nabla
_{e_{j}}^{E}\right\vert _{\Gamma _{b}\left( E\right) }.
\end{equation*}%
Then $D_{b}^{E}$ can be shown to be well-defined and is called the basic
Dirac operator corresponding to the foliated $\mathbb{C}\mathrm{l}\left(
Q\right) $ module $E$ (see \cite{GlK}). We note that this operator is not
symmetric unless a zero$^{\mathrm{th}}$ order term involving the mean
curvature is added; see \cite{KTFol}, \cite{KT4}, \cite{KT}, \cite{GlK}, 
\cite{PrRi}, \cite{HabRi}, \cite{BKR2} for more information regarding
essential self-adjointness of the modified operator and its spectrum. In the
formulas below, any lower order terms that preserve the basic sections may
be added without changing the index.

\begin{definition}
The \emph{analytic basic index} of $D_{b}^{E}$ is 
\begin{equation*}
\mathrm{ind}_{b}\left( D_{b}^{E}\right) =\dim \ker D_{b}^{E}-\dim \ker
\left( D_{b}^{E}\right) ^{\ast }.
\end{equation*}
\end{definition}

It is well-known that these dimensions are finite (see \cite{EKHS}, \cite{KT}%
, \cite{EK}, \cite{BKR2}), and it is possible to identify $\mathrm{ind}%
_{b}\left( D_{b}^{E}\right) $ with the invariant index of a first order, $%
O\left( q\right) $-equivariant differential operator $\widehat{D}$ over a
vector bundle over the basic manifold $\widehat{W}$. By applying the
equivariant index theorem (Theorem \ref{MainTheorem}), we obtain the
following formula for the index. In what follows, if $U$ denotes an open
subset of a stratum of $\left( M,\mathcal{F}\right) $, $U^{\prime }$ denotes
the desingularization of $U$ very similar to that in Section 4, 
and $\widetilde{U}$ denotes the fundamental domain of $U$ inside $U^{\prime }
$.

\begin{theorem}
(Basic Index Theorem for Riemannian foliations \cite{BKR2}) Let $M_{0}$ be
the principal stratum of the Riemannian foliation $\left( M,\mathcal{F}%
\right) $, and let $M_{1}$, ... , $M_{r}$ denote all the components of all
singular strata, corresponding to $O\left( q\right) $-isotropy types $\left[
G_{1}\right] $, ... ,$\left[ G_{r}\right] $ on the basic manifold. With
notation as in the discussion above, we have 
\begin{eqnarray*}
\mathrm{ind}_{b}\left( D_{b}^{E}\right) &=&\int_{\widetilde{M_{0}}\diagup 
\overline{\mathcal{F}}}A_{0,b}\left( x\right) ~\widetilde{\left\vert
dx\right\vert }+\sum_{j=1}^{r}\beta \left( M_{j}\right) ~ \\
\beta \left( M_{j}\right) &=&\frac{1}{2}\sum_{\tau }\frac{1}{n_{\tau }%
\mathrm{rank~}W^{\tau }}\left( -\eta \left( D_{j}^{S+,\tau }\right) +h\left(
D_{j}^{S+,\tau }\right) \right) \int_{\widetilde{M_{j}}\diagup \overline{%
\mathcal{F}}}A_{j,b}^{\tau }\left( x\right) ~\widetilde{\left\vert
dx\right\vert },
\end{eqnarray*}%
where the sum is over all components of singular strata and over all
canonical isotropy bundles $W^{\tau }$, only a finite number of which yield
nonzero terms $A_{j,b}^{\tau }$, and where

\begin{enumerate}
\item $A_{0,b}\left( x\right) $ is the Atiyah-Singer integrand, the local
supertrace of the ordinary heat kernel associated to the elliptic operator
induced from $\widetilde{D_{b}^{E}}$ (a desingularization of $D_{b}^{E}$) on
the quotient $\widetilde{M_{0}}\diagup \overline{\mathcal{F}}$, where the
bundle $E$ is replaced by the space of basic sections of over each leaf
closure;

\item $\eta \left( D_{j}^{S+,b}\right) $ and $h\left( D_{j}^{S+,b}\right) $
are defined in a similar way as in Theorem \ref{MainTheorem}, using a
decomposition $D_{b}^{E}=D_{N}\ast D_{M_{j}}$ at each singular stratum;

\item $A_{j,b}^{\tau }\left( x\right) $ is the local supertrace of the
ordinary heat kernel associated to the elliptic operator induced from $%
\left( \mathbf{1}\otimes D_{M_{j}}\right) ^{\prime }$ (blown-up and doubled
from $\mathbf{1}\otimes D_{M_{j}}$, the twist of $D_{M_{j}}$ by the
canonical isotropy bundle $W^{\tau }$) on the quotient $\widetilde{M_{j}}%
\diagup \overline{\mathcal{F}}$, where the bundle is replaced by the space
of basic sections over each leaf closure; and

\item $n_{\tau }$ is the number of different inequivalent $G_{j}$%
-representation types present in a typical fiber of $W^{\tau }$.
\end{enumerate}
\end{theorem}

\end{document}